\documentclass[conference]{IEEEtran}
\ifCLASSINFOpdf
\else
\fi

\usepackage{graphicx}
\usepackage{eurosym}
\usepackage{amssymb}
\usepackage{amsmath}
\usepackage{amsfonts}
\usepackage{epstopdf}
\usepackage{epsf,subfigure}
\usepackage{psfrag}
\usepackage{graphics}
\usepackage{color} 
\usepackage{cite}
\usepackage{array,multirow,pbox}
\usepackage{enumitem}

\newcommand{\myassum}[2]{
\begin{enumerate}[label=\textbf{A\arabic*}]
\setcounter{enumi}{\value{#1}}
#2
\setcounter{#1}{\value{enumi}}
\end{enumerate}}

\newcommand{\rem}[1]{}

\newtheorem{definition}{Definition}
\newtheorem{theorem}{Theorem}

\newtheorem{lemma}{Lemma}
\newtheorem{remark}{Remark}

\newenvironment{proof}[1][Proof]{\begin{trivlist}
\item[\hskip \labelsep {\bfseries #1}]}{\end{trivlist}}

\newcommand{\qed}{\nobreak \ifvmode \relax \else
      \ifdim\lastskip<1.5em \hskip-\lastskip
      \hskip1.5em plus0em minus0.5em \fi \nobreak
      \vrule height0.75em width0.5em depth0.25em\fi}

\graphicspath{{Figures/}}

\newcommand{\mysubeq}[2]{
\begin{subequations}\label{#1}
\begin{align}
#2
\end{align}\end{subequations}}

\begin{document}
%

\title{Assessment of Optimal Flexibility in Ensemble of Frequency Responsive Loads}



%
\author{\IEEEauthorblockN{Soumya Kundu,
Jacob Hansen, 
Jianming Lian and
Karan Kalsi
}
\IEEEauthorblockA{Optimization and Control Group\\
Pacific Northwest National Laboratory,
Richland, WA 99352, USA\\ Email: \{soumya.kundu,\,jacob.hansen,\,jianming.lian,\,karanjit.kalsi\}@pnnl.gov}
}


\maketitle

\begin{abstract}
Potential of electrical loads in providing grid ancillary services is often limited due to the uncertainties associated with the load behavior. A knowledge of the expected uncertainties with a load control program would invariably yield to better informed control policies, opening up the possibility of extracting the maximal load control potential without affecting grid operations. In the context of frequency responsive load control, a probabilistic uncertainty analysis framework is presented to quantify the expected error between the target and actual load response, under uncertainties in the load dynamics. A closed-form expression of an optimal demand flexibility, minimizing the expected error in actual and committed flexibility, is provided. Analytical results are validated through Monte Carlo simulations of ensembles of electric water heaters.
\end{abstract}


%
\IEEEpeerreviewmaketitle

\section{Introduction}
In the electrical power grid, any imbalance in the generation and load results in a change in the system frequency (excess supply increases frequency while excess demand reduces it). Grid operators employ various frequency control resources (e.g. speed governors, spinning reserves) that can act at different time-scales (e.g. sub-second to minutes) to arrest any changes in the frequency, and restore normalcy. With an increase in the penetration of renewable generation, the importance of adequate (and cost-effective) frequency response actions is expected to grow even further \cite{Strbac:2002,Short:2007}. Electrical loads can often provide a much faster, cleaner and less expensive alternative to the traditional frequency responsive resources. The use of demand flexibility for frequency response has been explored both in the academia and the industry \cite{Schweppe:1982,Kirby:2002,Callaway:2011, Kundu:2011PSCC, Perfumo:2012, Sinitsyn:2013, Mathieu:2013, Ma:13, Zhang:2013, Hao:14, Sanandaji:16}. 

Of particular interest to this article is the decentralized control of loads to provide primary frequency response. Traditionally generators are equipped with speed governors that use a `droop curve' to increase (/decrease) the mechanical torque into the generator rotor when the frequency is less (/greater) than desired, thereby increasing (/decreasing) the generator electrical power output. Because of the short time of response (of the order of 1-10\,s), any load control algorithm providing primary response is likely to be distributed (and, possibly, hierarchical) in nature, whereby each load in an ensemble monitors the grid frequency and decides to change its power consumption (e.g. switch on/off) autonomously. Note that a completely decentralized and autonomous response by the loads in an ensemble to a frequency event can easily lead to synchronization thereby causing potential instabilities in the grid. Thus often a hierarchical distributed control architecture is conceptualized in which a supervisor (e.g. a load aggregator) is tasked with dispersing the response of the loads across the ensemble so that some desirable collective behavior is attained. Dispersion of load response can be achieved in multiple ways, e.g. dispersion in time (by assigning to each load specific time-intervals in which to respond \cite{Moghadam:2013}), or dispersion in frequency (by assigning to each load specific frequency thresholds to respond to \cite{Lu:2006,Horst:2007,Molina_Garcia:2011,Lian:2016}). The frequency dispersion method has certain advantage as it can be designed to achieve certain power-frequency droop-like response thereby allowing easier integration of such frequency-responsive resources in the grid planning operation.

However, load behavior, unlike spinning reserves and other traditional frequency control resources, is usually uncertain and unpredictable. While spinning reserves are always ready to respond to frequency events, loads are expected to supply certain local demand, which determines their availability for frequency response. `Energy-driven' loads, such as any type of thermal loads (air-conditioners, electric water-heaters), for which the local demand is reliant on the energy consumption over a duration, offers greater flexibility and availability for frequency response. Availability of these loads to respond to frequency events is strongly influenced by their dynamics. A grid operator needs to be aware of such uncertainties regarding the load dynamic behavior in order to appropriately dispatch a mix of frequency responsive resources (spinning reserves, responsive load ensembles) during operations.  

In this article, we consider an ensemble of electric water-heaters offering frequency response services to the grid operator. In a hierarchical framework a load aggregator assigns to each participating device a frequency threshold in order to coordinate the ensemble response to frequency events. The goal of this article is to analyze the uncertainties in such schemes arising due to load dynamic behavior. Sec.\,\ref{S:system} describes the systems and the control problem. Sec.\,\ref{S:response} presents the a way to estimate the ensemble response under load uncertainties, while Sec.\,\ref{S:results} presents the numerical simulation results. Finally we conclude the article in Sec.\,\ref{S:concl}.

\section{System Description}\label{S:system}

\subsection{Load Model}
In this article, we consider electric water-heaters (EWHs) as the example of flexible loads providing frequency response. EWHs offer an attractive option for flexible demand response, due to several reasons, including 1) strong correlation between EWH demand profile and the usual daily load patterns; and 2) the relatively high percentage of domestic electrical loads that the EWHs represent. Depending on the requirements, the water temerature dynamics of an EWH can be modeled at varying details \cite{Xu:14}. For our purpose, it suffices to use the `one-mass' thermal model which assumes the temperature inside the water-tank is spatially uniform (valid when the tank is \textit{nearly} full or \textit{nearly} empty) \cite{Diao:2012}:
\begin{align}\label{E:x}
\dot{T}_w(t) &=-a(t)\,T_w(t)+b(s(t),t)\,,\\
\text{where, }a(t) &:=\frac{1}{C_w}\left(\dot{m}(t)\,C_p+W\right),\notag\\
\&~b(s(t),t)&:=\frac{1}{C_w}\left(s(t)\,Q_e+\dot{m}(t)\,C_p\,T_{in}(t)+W\,T_a(t)\right).\notag
\end{align}
$T_w$ denotes the temperature of the water in the tank, and $s(t)$ denotes a switching variable which determines whether the EWH is drawing power ($s(t)=1$ or `on') or not ($s(t)=0$ or `off'). The rest of the notations are listed in Table\,\ref{Tab:params}, along with (the range of) their typical values used in this paper. Unless otherwise specified, the parameters are assumed to be uniformly distributed in the given range of values, except the hot water-flow rate ($\dot{m}$) which is assumed to follow certain typical water draw profiles \cite{Diao:2012}. The state of the EWH (`on' or `off') is determined by the switching condition:
\begin{align}\label{E:s}
s(t^+)&=\left\lbrace \begin{array}{cl}
0\,, & \text{if }T_w(t)\geq T_{set}+\delta T/2\\
1\,, & \text{if }T_w(t)\leq T_{set}-\delta T/2\\
s(t)\,, & \text{otherwise}
\end{array}\right.,
\end{align}
where $T_{set}$ is the temperature set-point of the EWH with a deadband width of $\delta T$\,. The electric power consumed by the EWH is a function of its operational state, given by $s(t)P$ where $P$ denotes the (constant) power the EWH draws in its `on' state. 
Typical values of the parameters $T_{set},\,\delta T$ and $P$ are also listed in Table\,\ref{Tab:params}. 

\begin{table*}[thpb]
\caption{EWH Model Parameters}
\label{Tab:params}
\centering
\begin{tabular}{|*{4}{c|}}\hline
Parameter & Description & Value (Range) & Unit\\\hline
$T_a$ & room temperature &  $75\,\pm\,2.5$& [$^o$F] \\
$T_{in}$ & inlet water temperature &  $60\,\pm\,2.5$ & [$^o$F] \\
$T_{set}$ & temperature set-point & $130\,\pm\,5$ & [$^o$F]\\
$\delta T$ & width of temperature hysteresis deadband & $20$ & [$^o$F]\\
$C_w$ & thermal capacitance of the water in the tank &  $417.11$& [BTU/\,$^o$F] \\
$C_p$ & specific heat capacity of water &  $1$& [BTU/\,(lb-$^o$F)] \\
$W$ & thermal conductance of the tank shell &  $3\,\pm\,0.25$& [BTU/\,($^o$F-hr)] \\
$Q_e$ & heating capacity of the resistor &  $15360\,\pm\,1706$& [BTU/\,hr] \\
$\dot{m}$ & hot water flow rate & -- & [lb/\,hr]\\
$P$ & electric power consumed in `on' state & $4.5\,\pm\,0.5$ & [kW]\\\hline
\end{tabular}
\end{table*}

The EWH keeps switching between the two operational states - `on' and `off' - to maintain the water temperature within the specified temperature deadband. If the initial conditions and the parameters of the EWHs are randomly distributed, the amount of the time an EWH spends in the `on' state is also randomly distributed. 
\begin{definition}
Let us denote by $\Delta_{on}$  and $\Delta_{off}$ the random variables that represent the lengths of an `on' and `off' time-period, respectively.
\end{definition}

\subsection{Ensemble Frequency Response}
Consider an ensemble of $N$ number of EWHs, such that the $i$-th EWH, $i\!\in\!\lbrace 1,2,\dots,N\rbrace$\,, consumes power $s_i(t)\,P_i$ at any time $t$\,, where $s_i(t)$ and $P_i$ denote its switching variable and power rating, respectively. The total power consumed by the ensemble is 
\begin{align}
\forall t:~P_\Sigma^{}(t):={\sum}_{i=1}^Ns_i(t)\,P_i={\sum}_{\lbrace\forall i:s_i(t)=1\rbrace}P_i
\end{align}

When this ensemble commits to an under-frequency response, it is expected to decrease its power consumption by turning off some of its devices if the frequency falls. Since there are other frequency control mechanisms in-place, any such ensemble of loads will be expected to respond to events when the frequency is in a specific range. Thus a typical under-frequency response curve would look like Fig.\,\ref{F:droop}, where $\omega_u$ and $\omega_l$ denote the upper and lower limits of the frequency range assigned to the ensemble, and $\omega_0$ is the nominal frequency (60\,Hz). 
\begin{figure}[thpb]
\centering
\includegraphics[scale=0.3]{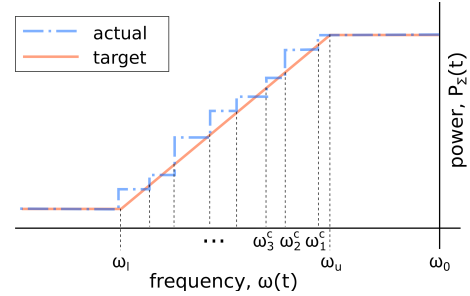}
\caption{Illustration of a power-frequency response curve.}
\label{F:droop}
\end{figure}
Clearly, $\omega_l<\omega_u\leq\omega_0\,.$ 
The target frequency response curve is a smooth line whose slope is determined based on the number (and power consumption) of the EWHs available to switch their states from `on' to `off'. The actual control is implemented by assigning frequency thresholds to each EWH, such that each EWH can turn `off' by monitoring the frequency on its own (see \cite{Molina_Garcia:2011,Lian:2016} for details). An over-frequency response policy can be constructed in a similar way. 

\begin{remark} From hereon, let us focus our discussion to the under-frequency response, with the understanding that an extension to the over-frequency would be trivial.
\end{remark}

In order to better explain the response policy, let us assume that any given time $t$\,, $\mathcal{S}_t=\lbrace d_1,\,d_2,\dots,\,d_{|\mathcal{S}_t|}\rbrace$ represents the set of indices of the `on' EWHs, while $|\mathcal{S}_t|$ denotes the number of `on' EWHs. Without any loss of generality, let us assume that the corresponding frequency thresholds $\left\lbrace \omega^c_i(t)\right\rbrace_{i=1}^{|\mathcal{S}_t|}$ are chosen in an ordered way so that,
\begin{align*}
\omega_l\leq\omega^c_{|\mathcal{S}_t|}(t)<\dots<\omega^c_2(t)<\omega^c_1(t)\leq\omega_u.
\end{align*} 
One possible way to choose the frequency thresholds to produce the target response curve in Fig.\,\ref{F:droop} is to assign 
\begin{align*}
\forall i\in\lbrace 1,\dots,|\mathcal{S}_t|\rbrace:~\omega^c_i(t):=\omega_u-\frac{\omega_u-\omega_l}{P_\Sigma^{}(t)}\sum_{j=1}^iP_{d_j}\,.
\end{align*}
The available EWHs obey the following response policy:
\begin{align}\label{E:response}
\forall i\in\lbrace 1,\dots,|\mathcal{S}_t|\rbrace:~s_{i}(t^+)&=0\,,~\text{ if }\omega(t)\leq\omega^c_i(t)\,,
\end{align}
where $s_{i}(t)=1\forall i\in\lbrace 1,\dots,|\mathcal{S}_t|\rbrace$\,. Note that we have ignored, for simplicity, any finite time-delay in the response. The total power consumption of the ensemble under this response policy, as illustrated in Fig.\,\ref{F:droop}, is given by,
\begin{align}
\forall t:~P_\Sigma(t^+)&=P_\Sigma^{}(t)-{\sum}_{\lbrace\forall d_i\in\mathcal{S}_t:\,\omega^c_i(t)\geq\omega(t)\rbrace}P_{d_i}\,.
\end{align}

The key point here is that the values of the frequency thresholds depend on the availability of the EWHs to turn `off' during a frequency event. However, continuous monitoring of the EWH states in an ensemble has high telemetry requirements, along with potential privacy concerns (for the EWH owners). A more viable option is to acquire and update the EWH states information once (at the start of) every fixed control time window, while using that information to estimate the availability of the responsive EWHs during the control window. Furthermore, if the control window is sufficiently short (say, 5-15\,min), the probability that there are more than one frequency events during a control window is negligibly small. Therefore, in this paper, we will focus on the scenario when the ensemble of EWHs will have to respond to at most one frequency event in each control window.

\subsection{Problem Statement}
Consider a control window $\mathcal{C}\!=\![t_0,\,t_f)$\,. At the start of the control window, $t=t_0\,,$ each EWH $i\in\lbrace 1,2,\dots,N\rbrace$ communicates to the load aggregator its power consumption $s_i(t_0)\,P_i$\,. Based on this information, the aggregator commits to the grid operator certain flexibility ($\overline{P_\Sigma^{}}$) for frequency response over the control window. For under-frequency response (Fig.\,\ref{F:droop}), this amounts to committing to reduce the aggregate power consumption by a maximum amount of $\overline{P_\Sigma^{}}$ over some frequency range $[\omega_l,\omega_u]$ in the form of a droop-curve. 

The objective of this paper is to determine the `optimal' value of the committed flexibility (denoted by $\overline{P_\Sigma^{}}^*$) which minimizes the maximal expectation of the squared relative difference between the actual and committed flexibility during a control window. Thus we seek the following:
\mysubeq{E:problem}{
\overline{P_\Sigma^{}}^*&:=\underset{\overline{P_\Sigma^{}}}{\arg\min}\left[\sup_{t\in\mathcal{C}}\mathbb{E}\left\lbrace \xi(t|t_0)^2\right\rbrace\right]\,,\\
\text{where, }\xi(t|t_0)&:=\frac{\left|P_\Sigma^{}(t)-\overline{P_\Sigma^{}}\right|}{\overline{P_\Sigma^{}}}\,.
}

\section{Optimal Flexibility}\label{S:response}

The set of EWHs that are in the `on'-state at the start of a control window $\mathcal{C}=[t_0,t_f)$ is given by $\mathcal{S}_{t_0}\subseteq\lbrace 1,2,\dots,N\rbrace$\,, with $|\mathcal{S}_{t_0}|$ being the number of `on' EWHs at $t=t_0$. Let us define the probability that a randomly selected EWH is `on' at any time $t\in\mathcal{C}$ by,
\begin{align}
\forall i:\quad p_{on}(t):=\text{Pr}\left(i\in\mathcal{S}_t\right)=\text{Pr}\left(s_i(t)=1\right)\quad\forall t\in\mathcal{C}\,.
\end{align}
Clearly, $p_{on}(t_0)=\left|\mathcal{S}_{t_0}\right|/N\,.$ Let us assume that,
\myassum{assum_list}{
\item\label{AS:similarity} The random variables representing the initial conditions and parameters of the EWHs are: 1) drawn from the same distribution, and 2) independent with each other. }
Based on \ref{AS:similarity} we argue that the natural (driven by unforced dynamics) `on' and `off' time-periods ($\Delta_{on}$ and $\Delta_{off}$\,, respectively) of each EWH in the ensemble follow the same probability density functions $f_{\Delta_{on}}\left(\cdot\right)$ and $f_{\Delta_{off}}\left(\cdot\right)$, respectively. Let us assume,
\myassum{assum_list}{
\item\label{AS:switching}
The length of the control window is sufficiently small such that the EWHs can change their state of operation at most only once during a control window.}
\begin{lemma}\label{L:pon} 
Probability that an EWH is `on' at $t\in\mathcal{C}$:
\begin{align}\label{E:pon}
\!\!p_{on}(t)\!=\!&\frac{|\mathcal{S}_{t_0}|}{N}\!-\!(t\!-\!t_0)\!\left[\alpha_{on}\frac{|\mathcal{S}_{t_0}|}{N}\!-\!\alpha_{off}\!\left(\!1\!-\!\frac{|\mathcal{S}_{t_0}|}{N}\!\right)\right],
\end{align}
where $\alpha_{on}\!:=\!\int_{t_f-t_0}^{\infty}\!\frac{f_{\Delta_{on}}(\tau)\,d\tau}{\tau}$ and $\alpha_{off}\!:=\!\int_{t_f-t_0}^{\infty}\!\frac{f_{\Delta_{off}}(\tau)\,d\tau}{\tau}\,$.
\end{lemma}
\begin{proof}
As per \ref{AS:switching}, there can be at most only one switching (`on'-to-`off' or `off'-to-`on'). Thus $p_{on}(t)$ is given by,
\begin{align*}
&p_{on}(t_0)\,\text{Pr}\left(i\!\in\!\mathcal{S}_t\left|\,i\in\mathcal{S}_{t_0}\right.\!\right)\!+\!\left(1\!-\!p_{on}(t_0)\right)\text{Pr}\left(i\!\in\!\mathcal{S}_t\left|\,i\notin\mathcal{S}_{t_0}\right.\!\right)\\
&\!=p_{on}(t_0)\cdot\text{Pr}\left(s_i(\tilde{t})\!=\!1\,\forall \tilde{t}\!\in\!(t_0,t]\left|\,i\in\mathcal{S}_{t_0}\right.\!\right)  \\
&+\left(1\!-\!p_{on}(t_0)\right)\!\cdot\!\left(1\!-\!\text{Pr}\left(s_i(\tilde{t})\!=\!0\,\forall \tilde{t}\!\in\!(t_0,t]\left|\,i\notin\mathcal{S}_{t_0}\right.\!\right)\right),
\end{align*}
where the first term denotes the probability that an EWH has been `on' throughout the time interval $[t_0,\,t]$\,; while the second term denotes the probability that an EWH was `off' at the start but has switched once from `off'-to-`on' during the interval $(t_0,\,t)\,.$ 

Let us denote by $\tau_{0,i}$ the random variable representing the length of time the $i$-th EWH ($i\in\mathcal{S}_{t_0}$) had already spent in the `on' state at the start of the control window. Based on the assumption\,\ref{AS:similarity}, it can be argued that the conditional distribution of $\tau_{0,i}$\,, given the `on' period $\Delta_{on,i}=\tau$\,, is uniform over the length of the `on' period, i.e.
\begin{align*}
\forall i:~\text{Pr}\left(\tau_{0,i}\leq v\left|\,\Delta_{on,i}=\tau\right.\right)={v}/{\tau}\quad\forall v\in[0,\tau]
\end{align*}
Therefore $\text{Pr}\left(s_i(\tilde{t})\!=\!1\,\forall \tilde{t}\!\in\!(t_0,t]\left|\,i\in\mathcal{S}_{t_0}\right.\!\right)$ is given by, 
\begin{align*}
&\int_{t_0-t_f}^\infty\!\!\text{Pr}\left(\tau_{0,i}\!+\!t\!-\!t_0\leq\tau\left|\Delta_{on,i}\!=\!\tau\right.\right)\,f_{\Delta_{on}}(\tau)\,d\tau\\
&=\int_{t_0-t_f}^\infty\!\!\frac{\tau-t+t_0}{\tau}\,f_{\Delta_{on}}(\tau)\,d\tau=1-(t-t_0)\,\alpha_{on}\,,
\end{align*}
where the lower integral limit follows from \ref{AS:switching}, i.e. $f_{\Delta_{on}}(\tau)=0\,\forall \tau<t_f-t_0$\,. Following similar arguments, we can argue that the probability $\text{Pr}\left(s_i(\tilde{t})\!=\!0\,\forall \tilde{t}\!\in\!(t_0,t]\left|\,i\notin\mathcal{S}_{t_0}\right.\!\right)=1-(t-t_0)\,\alpha_{off}$\,. This completes the proof.
\hfill\qed
\end{proof}

\begin{remark}
Note that the probability of an EWH being `on' is affine in time, with a slope that depends on the values of $\alpha_{on}$ and $\alpha_{off}$\,. While in some cases, with the exact knowledge of the EWH models and parameters, it may be possible to calculate $\alpha_{on}$ and $\alpha_{off}$ analytically, it is likely that their values would be estimated based on measurements. In this paper, we will assume that $\alpha_{on}$ and $\alpha_{off}$ are estimated online. 
\end{remark}

\begin{lemma}\label{L:metric}
The expected squared relative error between the actual and committed flexibility, for any $t\in\mathcal{C}$\,, is given by
\begin{align}\label{E:metric}
\!\!\mathbb{E}\lbrace \xi(t|t_0)^2\rbrace &\!=\! \frac{N}{N\!-\!1}\left[1\!-\!\frac{\left<P^2\right>}{2\overline{P_\Sigma^{}}\left<P\right>}\!-\!p_{on}(t)(N\!-\!1)\frac{\left<P\right>}{\overline{P_\Sigma^{}}}\right]^2\notag\\
&\quad -\frac{N}{N\!-\!1}\left(1\!-\!\frac{\left<P^2\right>}{2\overline{P_\Sigma^{}}\left<P\right>}\right)^2+1\,,
\end{align}
where $\left<P\right>:=\mathbb{E}\left[P_i\right]\,\forall i$ and $\left<P^2\right>:=\mathbb{E}\left[P_i^2\right]\,\forall i$\,.
\end{lemma}
\begin{proof}
Note that $\mathbb{E}\lbrace \xi(t|t_0)^2\rbrace\!=\!\frac{\mathbb{E}\lbrace P_\Sigma^{}(t)^2\rbrace}{\overline{P_\Sigma^{}}^2}-2\frac{\mathbb{E}\lbrace P_\Sigma^{}(t)\rbrace}{\overline{P_\Sigma^{}}}+1\,,$ where 
\begin{align*}
\mathbb{E}\lbrace P_\Sigma^{}(t)\rbrace &\!=\! {\sum}_{i=1}^N\mathbb{E}\lbrace s_i(t)\,P_i\rbrace\!=\!N\,p_{on}(t)\,\left<P\right>\,,\\
\&~\mathbb{E}\lbrace P_\Sigma^{}(t)^2\rbrace &\!=\!\mathbb{E}\left\lbrace  {\sum}_{i=1}^Ns_i(t)\,P_i^2+\!{\sum}_{i\neq j}s_i(t)s_j(t)P_iP_j\right\rbrace\\
&\!=\!N\,p_{on}(t)\left<P^2\right>+N(N\!-\!1)\,p_{on}(t)^2\left<P\right>^2.
\end{align*}
Here we use the fact that $s_i(t)^2=s_i(t)$\,. The rest of follows after simple algebraic manipulations.
\hfill\qed\end{proof}

\begin{lemma}
$\mathbb{E}\!\left\lbrace \xi(t|t_0)^2\right\rbrace\!\!\leq\!{\sup}_{t=\lbrace t_0,t_f^-\rbrace}\mathbb{E}\!\left\lbrace \xi(t|t_0)^2\right\rbrace\,\forall t\!\in\!\mathcal{C}$\,.
\end{lemma}
\begin{proof}
Note that ${\partial^2\mathbb{E}\left\lbrace \xi(t|t_0)^2\right\rbrace}/{\partial t^2}$ is given by:
\begin{align*}
\frac{2N(N\!-\!1)\left<P\right>^2}{\overline{P_\Sigma^{}}^2}\left(\alpha_{on}\frac{|\mathcal{S}_{t_0}|}{N}-\alpha_{off}\!\left(1\!-\!\frac{|\mathcal{S}_{t_0}|}{N}\right)\right)^2\geq 0\,.
\end{align*}
Therefore the maximum of $\mathbb{E}\left\lbrace \xi(t|t_0)^2\right\rbrace$ must occur either at $t=t_0$ or as $t\rightarrow t_f^-\,.$
\hfill\qed\end{proof}
We now present the main result of this paper:
\begin{theorem}\label{T:optimal}
In \eqref{E:problem}, the optimal value of $\overline{P_\Sigma^{}}$ is given by,
\begin{align}\label{E:solution}
\overline{P_\Sigma^{}}^*={\frac{\left<P^2\right>}{2\left<P\right>}+(N-1)\frac{(p_{on}(t_0)+p_{on}(t_f))}{2}\left<P\right>}{}.
\end{align}
\end{theorem}
\begin{proof}
We prove this by showing that the optimal solution of \eqref{E:problem} is attained when, for some $\overline{P_\Sigma^{}}$\,, the values of $\mathbb{E}\lbrace\xi(t|t_0)^2\rbrace$ at $t=t_0$ and $t=t_f^-$ are same. In order to see that, first note that twice differentiating $\mathbb{E}\lbrace\xi(t|t_0)^2\rbrace$ w.r.t. $\overline{P_\Sigma^{}}$\,, we get
\begin{align}\label{E:dEdP}
\frac{\partial^2\mathbb{E}\left\lbrace \xi(t|t_0)^2\right\rbrace}{\partial \overline{P_\Sigma^{}}^2}=\frac{2\,N\,p_{on}(t)\left<P\right>}{\overline{P_\Sigma^{}}^3}>0\,.
\end{align} 
For every $t$\,, there exists a unique value of $\overline{P_\Sigma^{}}$ such that,
\begin{align*}
\forall t:~\overline{P_\Sigma^{}}^t&:=\underset{\overline{P_\Sigma^{}}}{\arg\min}\left[\mathbb{E}\left\lbrace \xi(t|t_0)^2\right\rbrace\right]\\
&={\frac{\left<P^2\right>}{\left<P\right>}\!+\!(N\!-\!1)\,p_{on}(t)\!\left<P\right>}{}.
\end{align*}
Also note from Lemma\,\ref{L:pon} that $p_{on}(t)$ is affine in $t$\,, implying that the values of $\overline{P_\Sigma^{}}^t$ are distinct for every $t$\,. 

We will now argue that, for any choice of $\overline{P_\Sigma^{}}$ so that $\mathbb{E}\lbrace\xi(t_0|t_0)^2\rbrace\neq \mathbb{E}\lbrace\xi(t_f^-|t_0)^2\rbrace$\,, there exists a better choice of $\overline{P_\Sigma^{}}$\,. To do that, let us consider the following scenarios:
\begin{itemize}
\item \textsc{Case 1:} $\overline{P_\Sigma^{}}<\min\lbrace \overline{P_\Sigma^{}}^{t_0}\!,\,\overline{P_\Sigma^{}}^{t_f^-}\rbrace$\,. From \eqref{E:dEdP}, we see that as we increase $\overline{P_\Sigma^{}}$ towards $\min\lbrace \overline{P_\Sigma^{}}^{t_0}\!,\,\overline{P_\Sigma^{}}^{t_f^-}\rbrace$\,, the values of both $\mathbb{E}\lbrace\xi(t_0|t_0)^2\rbrace$ and $\mathbb{E}\lbrace\xi(t_f^-|t_0)^2\rbrace$ decrease monotonically. Therefore optimal value of $\overline{P_\Sigma^{}}$ must be at least as large as $\min\lbrace \overline{P_\Sigma^{}}^{t_0}\!,\,\overline{P_\Sigma^{}}^{t_f^-}\rbrace$\,.
\item \textsc{Case 2:} $\overline{P_\Sigma^{}}>\max\lbrace \overline{P_\Sigma^{}}^{t_0}\!,\,\overline{P_\Sigma^{}}^{t_f^-}\rbrace$\,. In a similar way as above, using \eqref{E:dEdP}, we can argue that the optimal value of $\overline{P_\Sigma^{}}$ cannot be larger than $\max\lbrace \overline{P_\Sigma^{}}^{t_0}\!,\,\overline{P_\Sigma^{}}^{t_f^-}\rbrace$\,.
\item \textsc{Case 3:} $\min\lbrace \overline{P_\Sigma^{}}^{t_0}\!,\,\overline{P_\Sigma^{}}^{t_f^-}\rbrace\leq\overline{P_\Sigma^{}}\leq\max\lbrace \overline{P_\Sigma^{}}^{t_0}\!,\,\overline{P_\Sigma^{}}^{t_f^-}\rbrace$\,. In such scenarios, if we increase (or, decrease) $\overline{P_\Sigma^{}}$\, the values of $\mathbb{E}\lbrace\xi(t_0|t_0)^2\rbrace$ and $\mathbb{E}\lbrace\xi(t_f^-|t_0)^2\rbrace$ change in the opposite directions. Because of continuity of the functions, therefore, we argue that the optimal value of $\overline{P_\Sigma^{}}$ lies between $\min\lbrace \overline{P_\Sigma^{}}^{t_0}\!,\,\overline{P_\Sigma^{}}^{t_f^-}\rbrace$ and $\max\lbrace \overline{P_\Sigma^{}}^{t_0}\!,\,\overline{P_\Sigma^{}}^{t_f^-}\rbrace$\,, for which $\mathbb{E}\lbrace\xi(t_0|t_0)^2\rbrace=\mathbb{E}\lbrace\xi(t_f^-|t_0)^2\rbrace$\,.
\end{itemize}
Based on the arguments above, the value of $\overline{P_\Sigma^{}}^*$ in \eqref{E:problem} is obtained by solving the following equation,
\begin{align*}
&\mathbb{E}\lbrace\xi(t_0|t_0)^2\rbrace=\mathbb{E}\lbrace\xi(t_f^-|t_0)^2\rbrace\\
\implies & \left[1\!-\!\frac{\left<P^2\right>}{2\overline{P_\Sigma^{}}^*\left<P\right>}\!-\!p_{on}(t_0)(N\!-\!1)\frac{\left<P\right>}{\overline{P_\Sigma^{}}^*}\right]^2\\
&\quad=\left[1\!-\!\frac{\left<P^2\right>}{2\overline{P_\Sigma^{}}^*\left<P\right>}\!-\!p_{on}(t_f)(N\!-\!1)\frac{\left<P\right>}{\overline{P_\Sigma^{}}^*}\right]^2\\
\implies &\overline{P_\Sigma^{}}^*={\frac{\left<P^2\right>}{2\left<P\right>}+(N-1)\frac{(p_{on}(t_0)+p_{on}(t_f))}{2}\left<P\right>}{}.
\end{align*}
This completes the proof.
\hfill\qed\end{proof}

\section{Numerical Results}\label{S:results}



Let us make some observation regarding the effect of water-flow rates on the dynamics. Since the inlet water temperature is lower than the temperature of the water in the tank (Table\,\ref{Tab:params}), the time an EWH spends in the `on' state increases as the inlet water-flow rate increases. In fact, if the water-flow rate is high enough then the temperature in the water-tank decreases even when in the `on' state, and is likely to fall below the hysteresis deadband $[T_{set}-\delta T/2,\,T_{set}+\delta T/2]$. Typical water-heater usage profiles, as used in \cite{Diao:2012}, are shown in Fig.\,\ref{F:waterflow}).
\begin{figure}[thpb]
\centering
\includegraphics[scale=0.35]{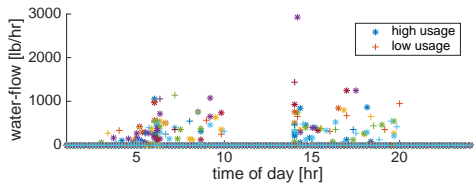}
\caption{Typical daily inlet water-flow profiles (high and low usage).}
\label{F:waterflow}
\end{figure}
\begin{figure*}[thpb]
\centering
\subfigure[$N=10\,,\,p_{on}(0)=1$]{
\includegraphics[scale=0.285]{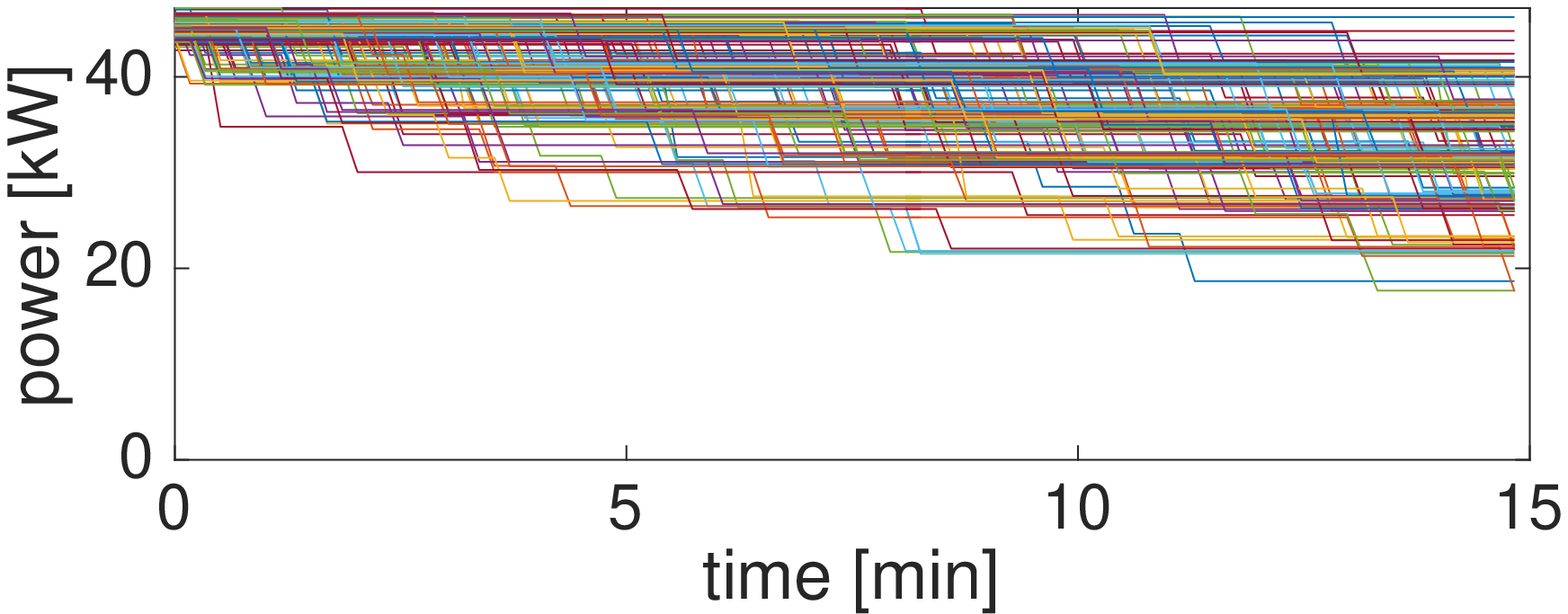}\label{F:10_100}
}\hspace{0.001in}
\subfigure[$N=200\,,\,p_{on}(0)=1$]{
\includegraphics[scale=0.285]{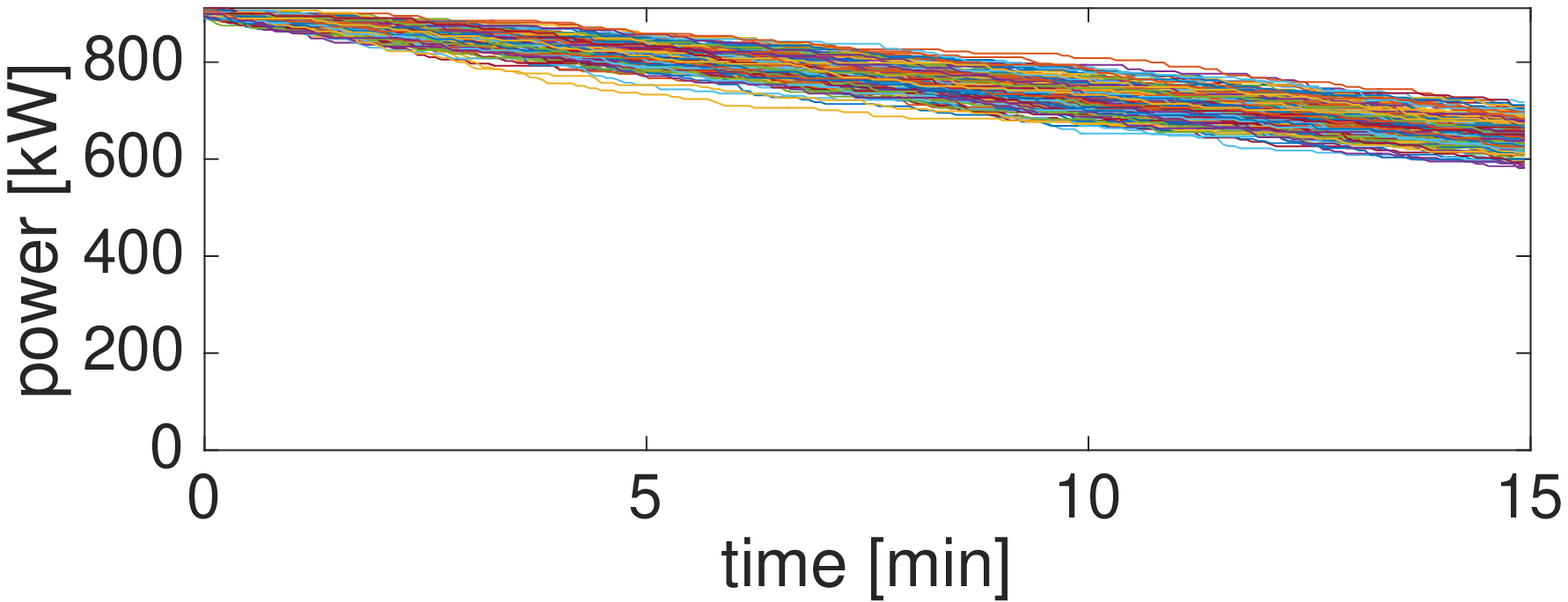}\label{F:200_100}
}\hspace{0.001in}
\subfigure[$N=1000\,,\,p_{on}(0)=1$]{
\includegraphics[scale=0.285]{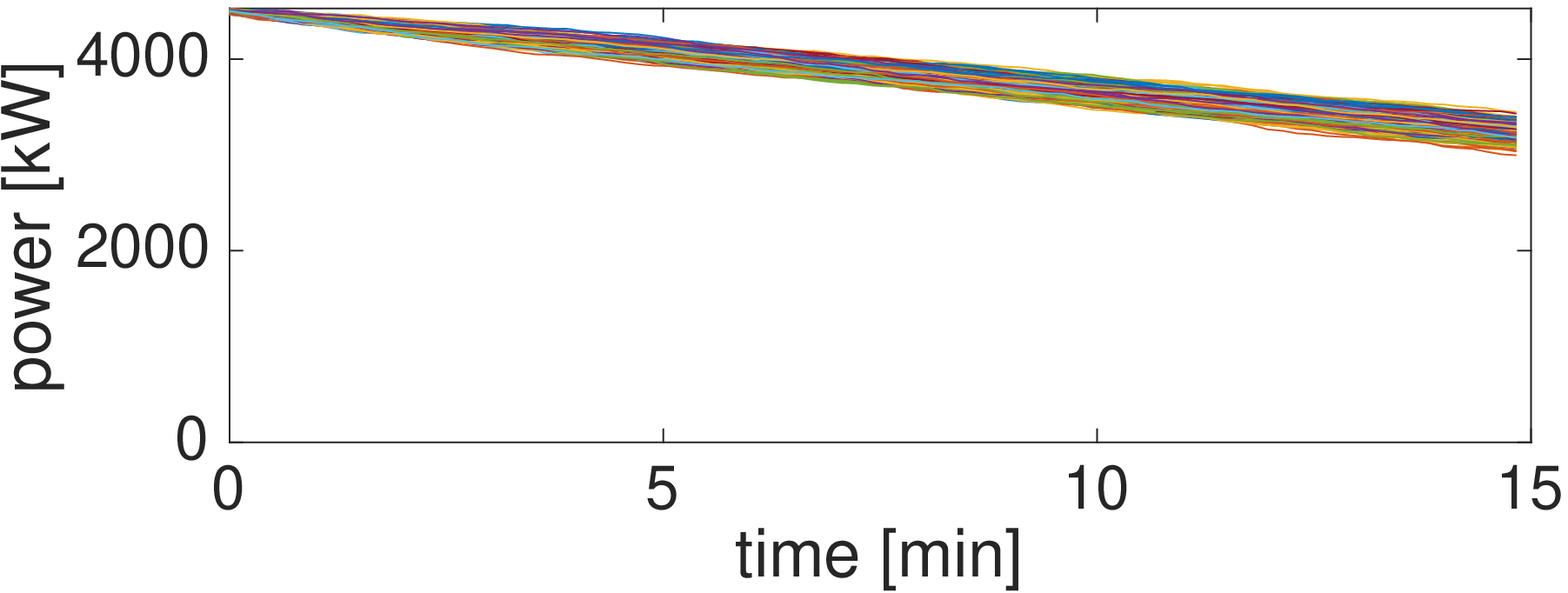}\label{F:1000_100}
}\hspace{0.001in}
\subfigure[$N=50\,,\,p_{on}(0)=1$]{
\includegraphics[scale=0.285]{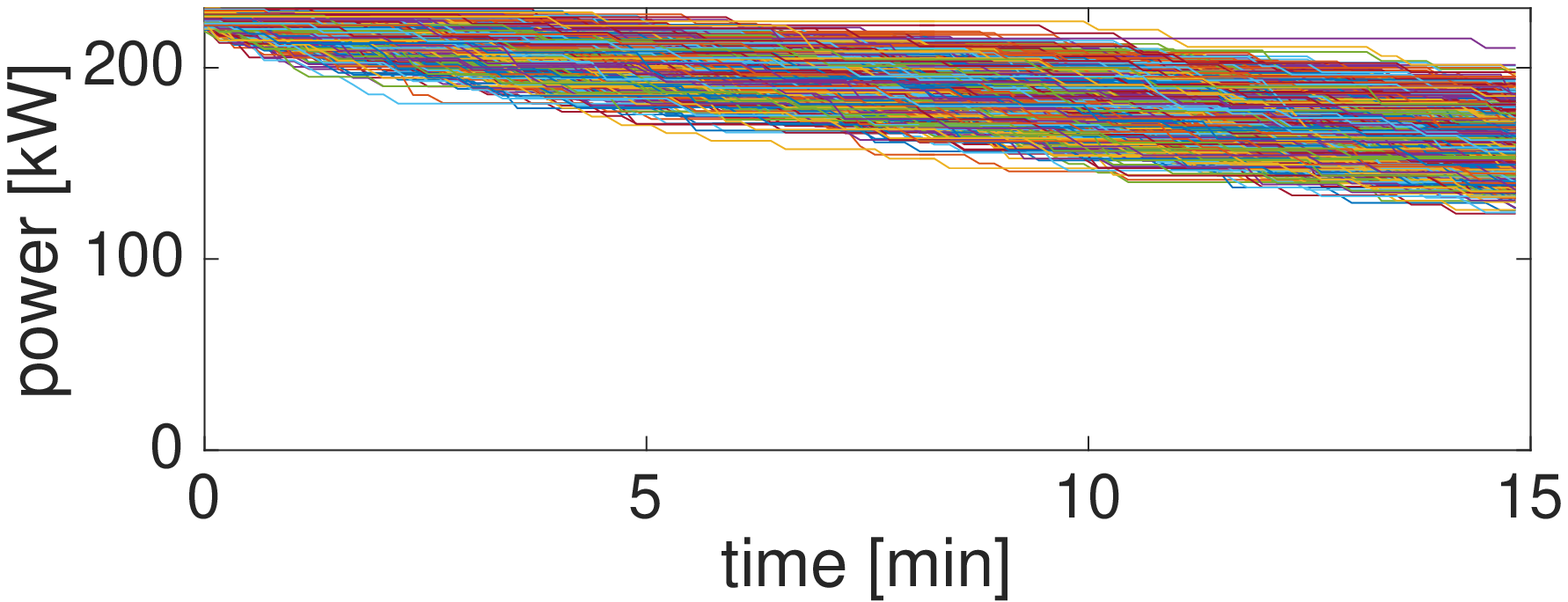}\label{F:50_100}
}\hspace{0.001in}
\subfigure[$N=50\,,\,p_{on}(0)=0.65$]{
\includegraphics[scale=0.285]{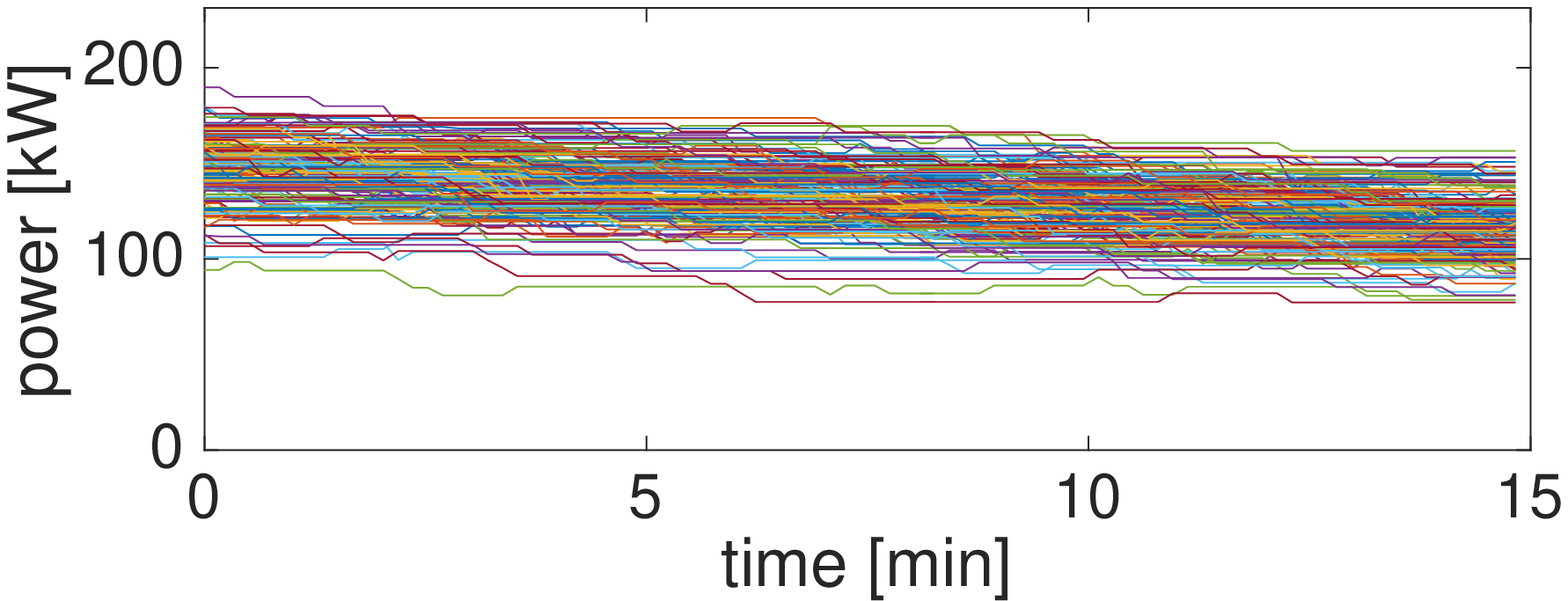}\label{F:50_65}
}\hspace{0.001in}
\subfigure[$N=50\,,\,p_{on}(0)=0.3$]{
\includegraphics[scale=0.285]{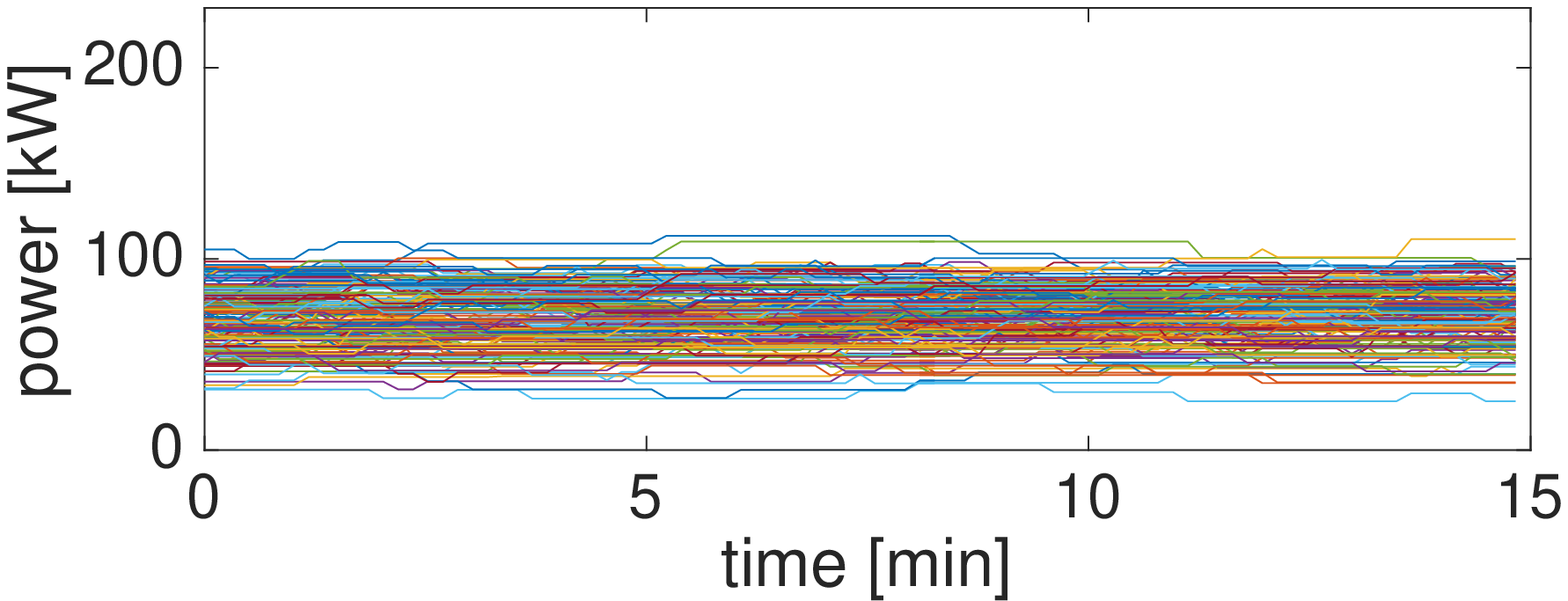}\label{F:50_30}
}\caption[Optional caption for list of figures]{Power consumption by ensembles of varying size (10,\,50,\,200,\,1000), with different initial fractions of `on' loads (0.3,\,0.65,\,1), for 200 instances each.}
\label{F:power}
\end{figure*}
\begin{figure*}[thpb]
\centering
\subfigure[$N=10\,,\,p_{on}(0)=1$]{
\includegraphics[scale=0.285]{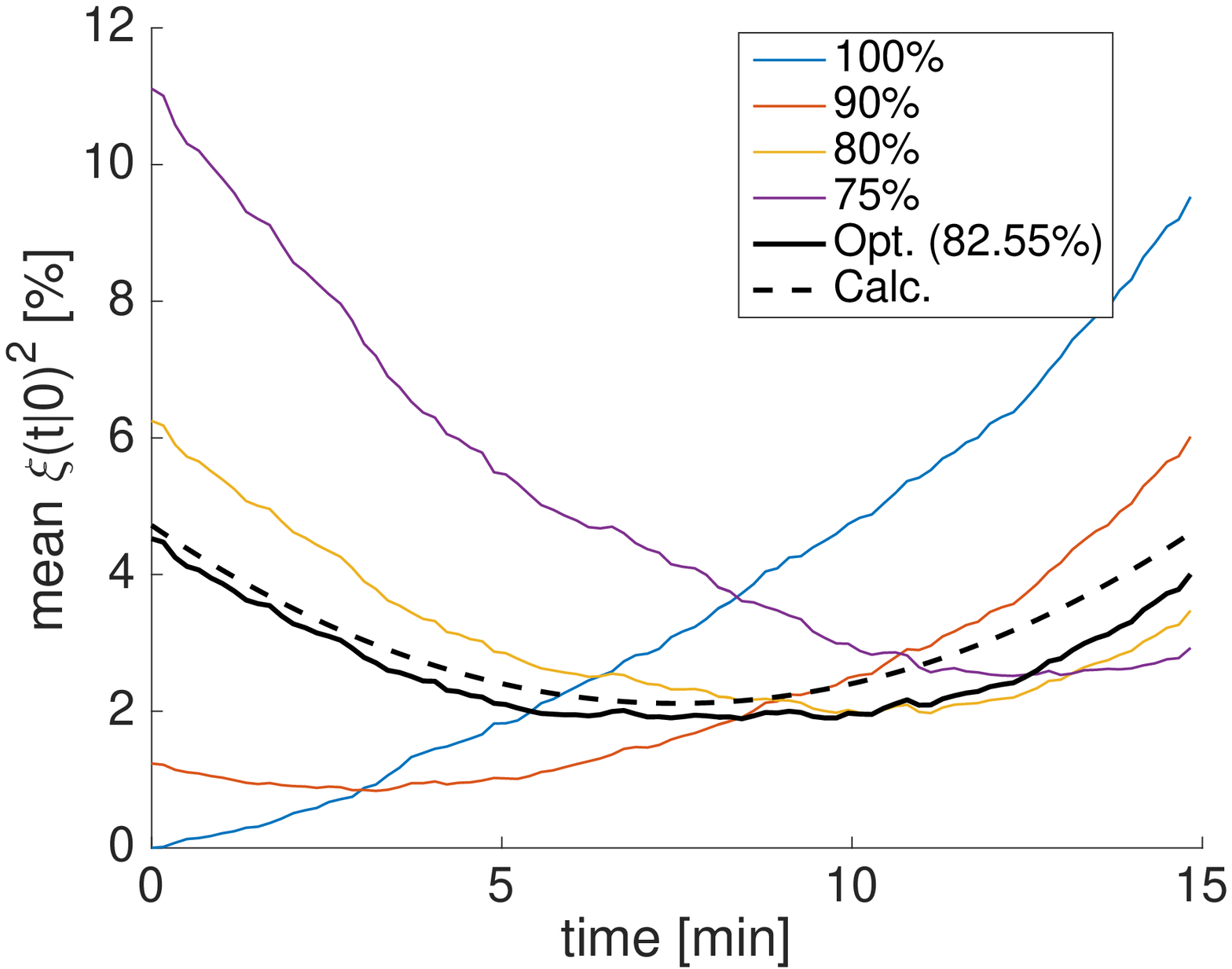}\label{F:error_10}
}\hspace{0.001in}
\subfigure[$N=100\,,\,p_{on}(0)=0.65$]{
\includegraphics[scale=0.285]{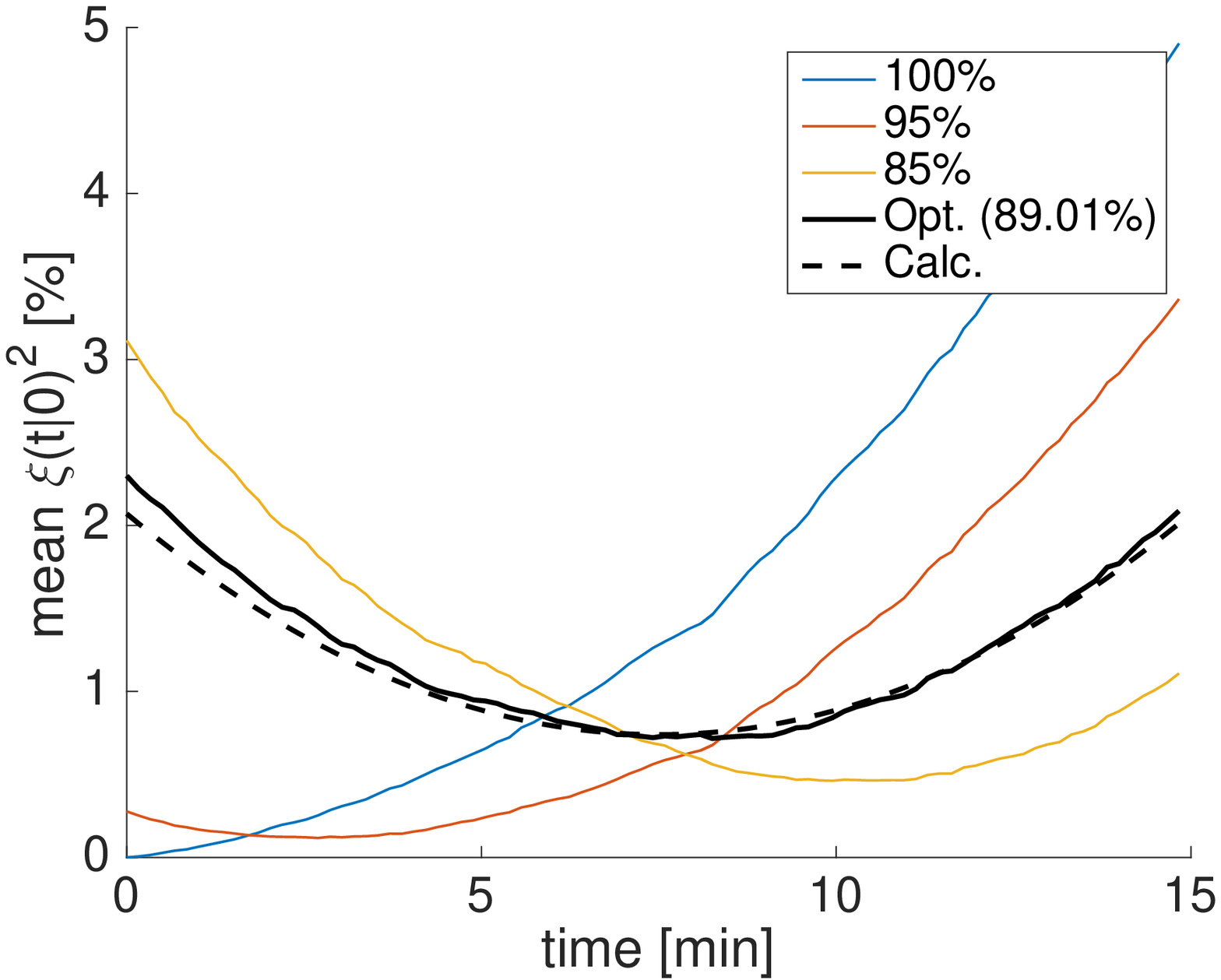}\label{F:error_100}
}\hspace{0.001in}
\subfigure[$N=1000\,,\,p_{on}(0)=1$]{
\includegraphics[scale=0.285]{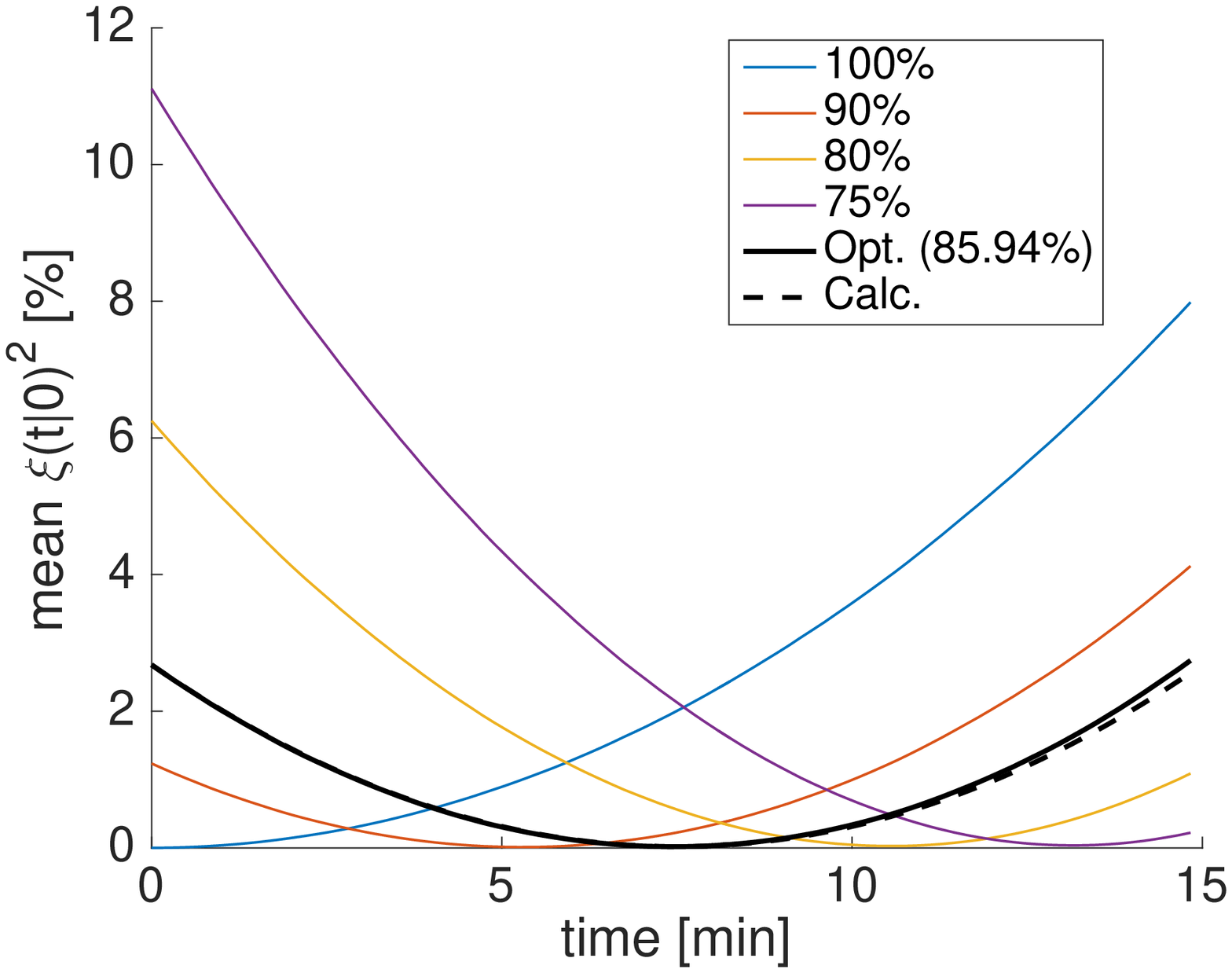}\label{F:error_1000}
}
\caption[]{Expected squared relative error, w.r.t. different levels of committed flexibility, compared with the optimal value of the flexibility commitment.}
\label{F:error}
\end{figure*}
For most of times (94.6\%) the water-flow rates are zero (less than 50\,lb/hr for 95\% of the time). Furthermore, the high water-flow rates are usually intermittent and do not sustain for long. Thus it is reasonable to assume that the EWHs commit for under-frequency response only when the is water-flow rates are very low (almost 95\% of time), and the water temperature in the tank is within the hysteresis deadband. Hence, for simplicity, we assume that, for the participating EWHs, the water-flow rates are zero and the tank water temperatures lie within the hysteresis deadband.


Fig.\,\ref{F:power} shows examples of how the ensemble size and the fractions of EWHs initially `on' affect the time evolution of total power consumption over a window of 15\,min. In Figs.\,\ref{F:10_100}-\ref{F:1000_100}, all the EWHs were initially `on', while the ensemble size was varied from 10 to 1000. Note that as the population size increases, the relative variability in the power consumption (w.r.t. the power consumption at $t\!=\!0$) decreases. Figs.\,\ref{F:50_100}-\ref{F:50_30} illustrate the evolution of total  power starting from varying fractions of EWHs that were initially `on'. When all EWHs were `on' initially, the total power decreases monotonically, but as the fraction of initially `on' EWHs decreases the total power decay rate reduces, even showing signs of increase in Fig.\,\ref{F:50_30}. 

\begin{figure}[thpb]
\centering
\includegraphics[scale=0.35]{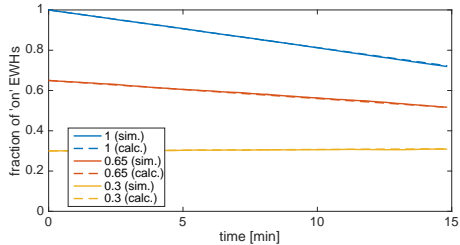}
\caption{Probability of an EWH being `on' at any time during a control window, for varying fractions of EWHs initially `on'. Generated from an ensemble size of 50, although the values are independent of the ensemble size.}
\label{F:probabilities}
\end{figure}

Based on these simulations, we can estimate the probability of an EWH being on at any time instant over the control window. Note from \eqref{E:pon} that the probability of an EWH being `on' at any point in time is a function of the fraction of EWHs that were `on' at the start and the distribution of `on' and `off' time-periods, but does not depend on the ensemble size. Fig.\,\ref{F:probabilities} shows the plot of evolution of the probability of an EWH being `on' ($p_{on}(t)$) at any time during a control window, for varying fractions (e.g. 1,\,0.65,\,0.3) of EWHs that are `on' at the start. The results are generated from an ensemble size of 50, although the curves are independent of the ensemble size (recall \eqref{E:pon}). Moreover, from these plots, we can estimate the values of $\alpha_{on}$ and $\alpha_{off}$\,, by first computing the value of $\alpha_{on}$ from the line corresponding to $p_{on}(0)=1$\,, and then use that value to compute $\alpha_{off}$ from either of the other two lines. Thus we calculate the following values,
\begin{align}\label{E:alpha}
\alpha_{on}=0.019\,\text{min}^{-1},~\alpha_{off}=0.009\,\text{min}^{-1}.
\end{align} 

Finally, we present some plots on the mean (expected) squared relative error, with respect to varying levels of flexibility commitment (towards frequency response) by an ensemble. Fig.\,\ref{F:error} shows the plots of mean squared relative errors at varying levels of commitment, for three different ensemble sizes (and different probabilities of being `on' at the start), computed numerically from 200 instances. From Fig.\,\ref{F:error_10}, we observe that, at an aggressive level of commitment at 100\%, the mean squared relative error monotonically increases from zero (at $t=0$) to close to 10\% towards the end of the control window (at $t=15\,$min). On the other hand, if the load aggregator bids conservatively at 75\%, the error value is high (over 10\%) at the start but monotonically decreases until the end of the control window. The same pattern is also observed in Figs.\,\ref{F:error_100}-\ref{F:error_1000}. Results are also shown for the optimal value of the committed flexibility (Theorem\,\ref{T:optimal}), calculated using the values of $\alpha_{on}$ and $\alpha_{off}$ from \ref{E:alpha}. Both the error values - 1) obtained from the analytical expression in \ref{L:metric} and 2) computed numerically from 200 instances, are shown in Fig.\,\ref{F:error}. We observe that the analytical and numerical values match closely. Further, the maximal error is lowest at the optimal flexibility level, at which level the error values at both ends of the control window are (almost) equal, as predicted in Theorem\,\ref{T:optimal}.

\section{Conclusion}\label{S:concl}
In this paper we discuss a hierarchical control framework whereby a load aggregator managing an ensemble of flexible EWHs commits certain frequency responsive reserve to the grid operator. At the start of a control window, each EWH communicates its state of operation (`on' or `off') to the load aggregator. The load aggregator's task is to estimate an optimal commitment level such that the maximal expected error between actual available reserve and the committed reserve over the control window is minimal. We provide a closed-form expression for the optimal flexibility that the load aggregator should commit for frequency response services. Simulation results are provided to validate the accuracy of the theoretical findings. Future work will focus on extending this analysis to wider class of flexible electrical loads.


\section*{Acknowledgment}

The authors would like to thank the United States Department of Energy for supporting this work under their Grid Modernization Lab Consortium initiative.



\bibliographystyle{IEEEtran}
\bibliography{IEEEabrv,RefList,MyReferences}
%
%
%

\end{document}